\documentclass[11pt,a4paper,english,reqno]{amsart}
\usepackage{amsmath,amssymb,amsfonts,epsfig,mathrsfs,hyperref}
\usepackage[T1]{fontenc}

\usepackage{color}
\usepackage{array}
\usepackage{amsthm}
\usepackage{amstext}
\usepackage{graphicx}
\usepackage{setspace}
\usepackage[margin=2.5cm]{geometry}
\usepackage{bbm}
\usepackage{color}
\usepackage{enumitem}
\allowdisplaybreaks[4]

\usepackage{amscd,psfrag}
\usepackage{yhmath}
\usepackage[mathscr]{eucal}

\usepackage{slashed}

\makeatletter
\pdfpageheight\paperheight
\pdfpagewidth\paperwidth

\usepackage{comment}

\usepackage{epstopdf}
\usepackage{chngcntr}
\usepackage{mathrsfs}

\usepackage{indentfirst}	

\usepackage[normalem]{ulem}
\theoremstyle{plain}

\newtheorem{definition}{Definition}[section]
\newtheorem{theorem}[definition]{Theorem}
\newtheorem*{theorem*}{Theorem}

\newtheorem{remark}[definition]{Remark}

\newtheorem*{remark*}{Remark}
\newtheorem*{sideremark*}{Side Remark}\newtheorem*{mt*}{Main Theorem}

\newtheorem*{claim*}{Claim}
\newtheorem*{q*}{Question}
\newtheorem{lemma}[definition]{Lemma}
\newtheorem{q}[definition]{Question}
\newtheorem{corollary}[definition]{Corollary}
\newtheorem*{corollary*}{Corollary}
\newtheorem*{proposition*}{Proposition}

\newcommand{\R}{\mathbb{R}}

\newcommand{\na}{\nabla}
\newcommand{\dd}{{\rm d}}
\newcommand{\p}{\partial}
\newcommand{\e}{\varepsilon}
\newcommand{\emb}{\hookrightarrow}
\newcommand{\weak}{\rightharpoonup}
\newcommand{\map}{\rightarrow}
\newcommand{\G}{\Gamma}
\newcommand{\M}{\mathcal{M}}

\newcommand{\1}{\mathbbm{1}}
\newcommand{\two}{{\rm II}}

\newcommand{\loc}{{\rm loc}}

\newcommand{\ttt}{{\mathfrak{t}}}
\newcommand{\nnn}{{\mathfrak{n}}}
\newcommand{\X}{{\mathscr{X}}}
\newcommand{\ball}{{\bf B}}

\allowdisplaybreaks[4]

\def\XXint#1#2#3{{\setbox0=\hbox{$#1{#2#3}{\int}$ }
\vcenter{\hbox{$#2#3$ }}\kern-.6\wd0}}

\newcommand{\so}{\mathfrak{so}}
\newcommand{\gl}{\mathfrak{gl}}
\newcommand{\g}{{\mathfrak{g}}}

\title{On the fundamental theorem of submanifold theory and isometric immersions with supercritical low regularity}

\author{Siran Li}

\address{Siran Li: School of Mathematical Sciences $\&$ CMA-Shanghai, Shanghai Jiao Tong University, No.~6 Science Buildings,
800 Dongchuan Road, Minhang District, Shanghai, China (200240)}

\email{\texttt{siran.li@sjtu.edu.cn}}

\author{Xiangxiang Su}
\address{Xiangxiang Su: School of Mathematical Sciences, Shanghai Jiao Tong University, No.~6 Science Buildings,
800 Dongchuan Road, Minhang District, Shanghai, China (200240)}
\email{\texttt{sjtusxx@sjtu.edu.cn}}


\date{\today}

\begin{document}

\maketitle

\begin{abstract}
A fundamental result in global analysis and nonlinear elasticity asserts that given a solution $\mathfrak{S}$ to the Gauss--Codazzi--Ricci equations over a simply-connected closed  manifold $(\M^n,g)$, one may find an isometric immersion $\iota$ of $(\M^n,g)$ into the Euclidean space $\R^{n+k}$ whose extrinsic geometry coincides with $\mathfrak{S}$. Here the dimension $n$ and the codimension $k$ are arbitrary.  Abundant literature has been devoted to relaxing the regularity assumptions on $\mathfrak{S}$ and $\iota$. The best result up to date is $\mathfrak{S} \in L^p$ and $\iota \in W^{2,p}$ for $p>n \geq 3$ or $p=n=2$. 

In this paper, we extend the above result to $\iota \in \mathcal{X}$ whose topology is strictly weaker than $W^{2,n}$ for $n \geq 3$. Indeed, $\mathcal{X}$ is the weak Morrey space $L^{p, n-p}_{2,w}$ with arbitrary $p \in ]2,n]$. This appears to be first supercritical result in the literature on the existence of isometric immersions with low regularity, given the solubility of the Gauss--Codazzi--Ricci equations. Our proof essentially utilises the theory of Uhlenbeck gauges --- in particular, Rivi\`{e}re--Struwe's work \cite{rs} on harmonic maps in arbitrary dimensions and codimensions --- and  compensated compactness. 

\end{abstract}

\section{Introduction}\label{sec: intro}

\subsection{Submanifold theory} 

The \emph{fundamental theorem of surface theory}, a classical result in differential geometry and global analysis, asserts that given a solution $\two$ (a $2 \times 2$ symmetric matrix field on $\M$) to the Gauss--Codazzi equations on any simply-connected closed Riemannian manifold $(\M,g)$, one may construct an isometric immersion $\iota: (\M,g) \to \R^3$ whose second fundamental form is precisely $\two$. The Gauss--Codazzi equations constitute necessary conditions for the existence of isometric immersions, while the fundamental theorem of surface theory  in turn ascertains the sufficiency. Proofs of the local version of the theorem can be found in  Choquet-Bruhat, Dewitt-Morette, and Dillard-Bleick \cite[p. 303]{add} and  Malliavin \cite[p. 133]{add'}.

Mathematical investigations for the fundamental theorem of surface theory have been largely motivated by nonlinear elasticity. One of the major objectives of elasticity theory is to determine the \emph{deformation} undergone by  elastic bodies in response to external forces and boundary conditions. In geometric terms, an elastic body is modelled by a 2- or 3-dimensional manifold $\M$, and the deformation is an isometric immersion $\iota: \M \emb \R^3$. The intrinsic approach to nonlinear elasticity, pioneered by Antman \cite{antman} and Ciarlet \cite{cbook, cbook2}, recasts the problems concerning the deformation $\iota$ to those concerning the Cauchy--Green tensor, which is the pullback Riemannian metric $\iota^\#\delta$ on $\M$ (where $\delta$ denotes the Euclidean metric throughout).

The generalisation of the fundamental theorem of surface theory to higher dimensions and codimensions, termed as the \emph{fundamental theorem of submanifold theory}, has also long been a folklore theorem. It ascertains that given a solution $\mathfrak{S}$ to the Gauss--Codazzi--Ricci equations on a simply-connected closed Riemannian manifold $(\M^n,g)$, one can find an isometric immersion $\iota: (\M^n,g) \to (\R^{n+k},\delta)$ whose extrinsic geometry coincides with $\mathfrak{S}$. Here and hereafter
\begin{align}\label{frakS, new}
\mathfrak{S} = \left(\two, \na^\perp\right),
\end{align}
where $\two$ is the second fundamental form and $\na^\perp$ is the normal connection, \emph{i.e.}, the orthogonal projection of the Levi-Civita connection on $\R^{n+k}$ to the normal directions of $\iota(\M)$. See Tenenblat \cite{ten} for a proof of the fundamental theorem of submanifold theory in the case that $\mathfrak{S}, \iota \in C^\infty$.

Lasting endeavours have been devoted over the past decades to lowering the regularity assumptions. Denote the regularity class of $\iota$ as $\X$. As aforementioned, given a solution $\mathfrak{S}$ to the Gauss--Codazzi--Ricci equations on the simply-connected closed manifold $(\M,g)$, Tenenblat \cite{ten} constructed an isometric immersion $\iota \in \X = C^\infty$. Arguments in \cite{hw} by Hartman--Wintner led  in effect to the existence of $\iota \in \X = C^3$ given $\mathfrak{S}\in C^1$. See 
also Ciarlet--Larsonneur \cite{ciarlet'} and the elasticity work by Blume \cite{b} for a careful treatment of the case $n=2$, $k=1$. In 2003, S. Mardare \cite{m1} extended the fundamental theorem of surface theory to $\X = W^{2,\infty}$ given $\mathfrak{S} \in L^\infty$. Then, in 2005, the same author \cite{m2} further settled the case $\X = W^{2,p}$ given $\mathfrak{S} \in L^p$ for arbitrary $p>2$. The higher dimensional and codimensional analogue of \cite{m2} was obtained in 2008 by Szopos \cite{sz} (see also S. Mardare \cite{cl, li-new, m3} for geometrically oriented variants of the arguments); that is, for any $n \geq 2$ one can find an isometric immersion of $(\M^n,g)$ into $\R^{n+k}$ in regularity class $\X = W^{2,p}$,  given $\mathfrak{S} \in L^p$ for arbitrary $p>n$.

\subsection{The critical $W^{2,n}$-isometric immersions}\label{subsec: intro, pfaff}
In view of the discussions above, it is natural to ask about the critical case of \cite{m2, sz}, which remains open up to date except for  $n=2$:
\begin{q}\label{q, fund}
Given a solution $\mathfrak{S} \in L^n$ on a simply-connected closed manifold $(\M^n,g)$, does there exist an isometric immersion $\iota\in W^{2,n}\left(\M,\R^{n+k}\right)$ whose extrinsic geometry is  $\mathfrak{S}$?
\end{q}

The criticality of $p=n$ has been pointed out in \cite{m2} based on the heuristics below. To construct $\iota$, one may solve first a \emph{Pfaff} system
\begin{align}\label{pfaff, new}
dP + \Omega P =0 
\end{align}
and then a \emph{Poincar\'{e}} system 
\begin{equation}\label{poincare, new}
d\iota = \omega P,
\end{equation}
where $\Omega \in \G\left(T^*\M \otimes \so(n+k)\right)$ is the connection 1-form of the Cartan formalism for the immersed submanifold (see \S\ref{subsec: cartan} below), $P: \M \to SO(n+k)$ represents a global change of coordinates, and $\omega:=\big(\omega^1,\ldots,\omega^n,0,\ldots,0\big)^\top$ where $\left\{\omega^i\right\}$ is a local orthonormal coframe on $T^*\M$. The Pfaff system~\eqref{pfaff, new} can be formally solved, treating everything as scalars,  by $$P \sim e^{-\int\Omega}.$$ The exponential function on the right-hand side is locally integrable only when the integral of $\Omega$ is in $L^\infty_\loc$. Hence, in view of the Sobolev--Morrey embedding $W^{1,p}_\loc(\M^n)\emb L^\infty_\loc(\M^n)$ for $p>n$, it is natural to require that $\Omega \in L^p$ for $p>n$ in \cite{m2, m3, sz}.

In 2020, Litzinger \cite{l} answered Question~\ref{q, fund} affirmatively in the critical case $\X = W^{2,2}$. He established the fundamental theorem of surface theory for $L^2$-second fundamental forms and $W^{2,2}$-isometric immersions, aligning with the regularity setting of the Willmore energy. As argued in \cite[\S 1, paragraph ensuing Theorem~1]{l}, the critical regularity case $p=n=2$ is attained by way of exploring the \emph{antisymmetry} of $\Omega$ and utilising the theory of Uhlenbeck gauges as well as Wente's compensated compactness estimate \cite{w}.

\subsection{Our main result}
We extend the result in \cite{l} to arbitrary dimensions and codimensions, by way of further leveraging the antisymmetric structure of $\Omega$ and putting it into the Coulomb--Uhlenbeck gauge. The idea arises from Uhlenbeck's seminal work \cite{u}: there exists a \emph{gauge transform} $P: \M \to SO(n+k)$ that takes $\Omega$ to a coexact (\emph{i.e.}, divergence-free) 1-form while preserving the flat curvature. The gauge transform can be regarded as a \emph{global} change of coordinates. The new coordinate frame proves favourable from analytic perspectives, as ellipticity emerges in the PDE for $P^\#\Omega$, the gauge-transformed connection 1-form. The theory of compensated compactness (Coifman--Lions--Meyer--Semmes \cite{clms}) also plays a key role in our analysis.

In this work, we are able to answer Question~\ref{q, fund} in the affirmative in the critical case $\X = W^{2,n}$ in any dimension $n \geq 3$ and codimension $k \geq 0$:
\begin{theorem}\label{thm: main, crude}
Let $(\M^n,g)$ be a simply-connected closed Riemannian manifold. Suppose that $\mathfrak{S} \in L^n$ is a weak solution to the Gauss--Codazzi--Ricci equations on $\M$. There exists a $W^{2,n}$-isometric immersion $\iota: (\M,g) \emb (\R^{n+k},\delta)$ whose extrinsic geometry coincides with $\mathfrak{S}$. 
\end{theorem}

In fact, we may go further beyond the critical regularity space $W^{2,n}$. The following stronger result, Theorem~\ref{thm: main, new}, will be established as our main theorem. Here and hereafter, $L^{p,n-p}_w$ denotes the weak Morrey space (see \S\ref{sec: nomenclature} for details), for which we have the continuous embeddings $$L^{n}(\M) \emb L^{p,n-p}_w(\M) \emb L^2(\M) \qquad \text{for any } 2<p\leq n$$ over compact manifold $\M^n$. Hence, Theorem~\ref{thm: main, crude} follows immediately from Theorem~\ref{thm: main, new}.

\begin{theorem}\label{thm: main, new}
Let $(\M^n,g)$ be a simply-connected closed Riemannian manifold; $n \geq 3$. Suppose that $\mathfrak{S} \in L^{p,n-p}_w$ with $2<p \leq n$ is a weak solution to the Gauss--Codazzi--Ricci equations on $\M$ with arbitrary codimension $k$. There exists an isometric immersion $\iota: (\M,g) \emb (\R^{n+k},\delta)$ in the regularity class $\X = L^{p,n-p}_{2,w}$ whose extrinsic geometry coincides with $\mathfrak{S}$.  Moreover, $\iota$ is unique modulo Euclidean rigid motions in $\R^{n+k}$ outside null sets. 
\end{theorem}

A direct consequence of Theorem~\ref{thm: main, new} is the weak compactness of immersions in the regularity class $\X = L^{p,n-p}_{2,w,\loc}$, $2 < p \leq n$, in line with \cite{cl, l}: A uniformly bounded family of immersions in $\X$ converges weakly to an immersion, with its induced metrics weakly convergent in $L^{p,n-p}_{1,w,\loc}$ and extrinsic geometries weakly convergent in $L^{p,n-p}_{w,\loc}$ modulo subsequences, \emph{provided that the induced metrics are nondegenerate}. See \S\ref{sec: weak compactness},  Theorem~\ref{thm: weak rigidity} for details.

In the remaining parts of the introduction, we outline the key steps of the proof. 

\subsection{Gauss--Codazzi--Ricci Equations and Cartan formalism}\label{subsec: cartan}

We briefly recall the Gauss--Codazzi--Ricci equations, which are compatibility equations for the existence of isometric immersions. A map $\iota: (\M,g)\map (\R^{n+k}, \delta)$ is said to be an \emph{isometric immersion} if and only if $d\iota$ is one-to-one and $g=\iota^\#\delta$, both in the \emph{a.e.} sense with respect to the Riemannian measure.

Given a Riemannian manifold $(\M,g)$ of dimension $n \geq 2$ and a vector bundle $\pi: E \to \M$ of rank $k$ with some Riemannian metric $g^E$ and metric-compatible connection $\na^E$. One may write down the Gauss, Codazzi, and Ricci equations associated to the bundle $E$ by splitting the zero Riemann curvature tensor into $(T\M, T\M)$, $(T\M, E)$, and $(E,E)$-factors, respectively. Here $E$ is the putative normal bundle --- once an isometric immersion $\iota: (\M,g) \to (\R^{n+k},\delta)$ has been found,  $T^\perp \M := T\R^{n+k}\slash T[\iota(\M)]$ can be identified with $E$ via a bundle isomorphism. 

In addition, assume that there exists a bilinear mapping
\begin{equation*}
\mathcal{S}: \G(E) \times \G(T\M) \longrightarrow \G(T\M), \qquad \mathcal{S}(\eta, X) \equiv  \mathcal{S}_\eta X 
\end{equation*}
such that 
\begin{equation*}
g\big(X, \mathcal{S}_\eta Y\big) = g\big( \mathcal{S}_\eta X, Y\big)
\end{equation*}
for all $X,Y \in \G(T\M)$ and $\eta \in \G(E)$. We shall then define the tensor $$\two:\G(T\M) \times \G(T\M) \longrightarrow \G(E)$$ by 
\begin{equation*}
g^E\big(\two(X,Y),\eta\big) := -g\big(\mathcal{S}_\eta X, Y\big).
\end{equation*}

With the above notations at hand, the Gauss, Codazzi, and Ricci equations may be expressed respectively as follows. See do Carmo \cite[Chapter~6]{d} and Tenenblat \cite{ten}.
\begin{eqnarray}
&& g\big(\two(X,Z), \two (Y,W)\big) - g\big( \two(X,W),\two(Y,Z) \big) = R(X,Y,Z,W),\label{gauss}\\
&& \overline{\na}_Y\two(X,Z) - \overline{\na}_X\two(Y,Z)=0,\label{codazzi}\\
&& g\big([\mathcal{S}_\eta, \mathcal{S}_{\zeta}] X, Y\big) = R^E(X,Y,\eta,\zeta),\label{ricci}
\end{eqnarray}
for all $X,Y,Z,W \in \G(T\M)$ and $\eta,\zeta \in \G(E)$. Here $[\bullet,\bullet]$ is the commutator,  $R$ and $R^E$ are the Riemann curvature tensors for $(T\M,g)$ and $\left(E,g^E\right)$, respectively, and $\overline{\na}$ is the Levi-Civita connection on the Euclidean space $\left(\R^{n+k},\delta\right)$.

From the PDE perspectives, in Equations~\eqref{gauss}, \eqref{codazzi}, and \eqref{ricci}, the Riemannian metric $g$ is given and the extrinsic geometry $\left(\two, \na^E\right)$ are unknown. Thus,  the solution $\mathfrak{S}$ to the Gauss--Codazzi--Ricci equations in Theorem~\ref{thm: main, new} is understood as $\mathfrak{S} = \left(\two, \na^E\right)$\footnote{The slight difference between the notations here and in \eqref{frakS, new} arises due to the following subtle reason: before finding the isometric immersion $\iota$ via solving the Gauss--Codazzi--Ricci equations, one cannot define \emph{a priori} the normal bundle. Thus, the Gauss--Codazzi--Ricci equations are formulated instead with respect to an abstract vector bundle $E$, which shall be later justified to coincide with the normal bundle once $\iota$ is solved. This is also the approach taken by Tenenblat \cite{ten}.}.

To proceed, let $\{\p\slash\p_i\}_{i=1}^n$ be a local moving frame for $T\M$ on some chart $U\subset \M$ such that $E$ is trivialised over $U$, and let $\{\eta_\alpha\}_{\alpha=n+1}^{n+k}$ be a local moving frame for $E\big|_U$. We follow \cite{ten} for convention on indices: $1\leq i,j,k\leq n$; $n+1\leq\alpha,\beta \leq n+k$; and $1 \leq a,b,c \leq n+k$. That is, $i,j,k$ index for the tangent bundle $T\M$, and $\alpha,\beta$ index for the putative normal bundle $E$.

The \emph{Cartan formalism} refers to the following identities in local coordinates:
\begin{eqnarray}
&&d\omega^i = \sum_j \omega^j \wedge \Omega^i_j;\label{first structure eq}\\
&&0=d\Omega^a_b + \sum_c \Omega^c_b\wedge \Omega^a_c,\label{second structure eq}
\end{eqnarray}
where $\{\omega^i\}_{1\leq i \leq n}$ is the orthonormal coframe on $(T^*\M,g)$ dual to $\{\p\slash\p_i\}_{i=1}^n$, and $\{\Omega^a_b\}_{1\leq a,b \leq n+k}$ is defined entry-wise by
\begin{eqnarray}
&& \Omega^i_j (\p_k) := g(\na_{\p_k}\p_i,\p_j);\label{Omega 1}\\
&& \Omega^i_\alpha (\p_j) \equiv -\Omega^\alpha_i (\p_j) := g^E\big(\two(\p_i,\p_j), \eta_\alpha\big);\label{Omega 2}\\
&& \Omega^\alpha_\beta(\p_j):=g^E\big(\na^E_{\p_j}\eta_\alpha,\eta_\beta\big).\label{Omega 3}
\end{eqnarray}
In the above, $\{\p_i\}$ is the orthonormal frame for $(T\M,g)$ dual to $\{\omega^i\}$. In the Cartan formalism, $\Omega=\{\Omega^a_b\}$ is said to be the \emph{connection 1-form}. 

The Cartan formalism equations~\eqref{first structure eq} and \eqref{second structure eq}, \emph{a.k.a.} the first and second \emph{structural equations}, are well-known to be equivalent to the  Gauss--Codazzi--Ricci Equations~\eqref{gauss}--\eqref{ricci} as purely algebraic (pointwise) identities, hence can be easily validated in the sense of distributions too. 

\subsection{Isometric immersions in good global coordinates}
We now change our perspectives --- instead of working in local coordinates, we consider the second structural equation~\eqref{second structure eq} as a global identity on the frame bundle over $(\M,g)$. In this subsection let us only highlight the key ideas; detailed explanations for the gauge-theoretic terminologies are presented in \S\ref{sec: gauge}.

As pointed out in \S\ref{subsec: intro, pfaff}, the key to the construction of an isometric immersion $\iota: (\M,g) \emb (\R^{n+k},\delta)$ is to solve for a \emph{gauge transform}, \emph{i.e.}, a global change of coordinates $$P: \M \longrightarrow SO(n+k),$$ which satisfies the Pfaff Equation~\eqref{pfaff, new}: $$dP+\Omega P=0.$$ Note that $P$ acts on the curvature 2-form (see \S\ref{sec: gauge},  \eqref{def, curvature} below) as follows:
\begin{equation}\label{gauged curvature}
 0 = P^\#\mathfrak{F}_\Omega = d\left(P^\#\Omega\right) + \left(P^\#\Omega\right)\wedge\left(P^\#\Omega\right),
\end{equation}
where
\begin{equation}\label{gauged connection 1-form}
P^\#\Omega := P^{-1}dP + P^{-1}\Omega P.
\end{equation}
Since $P^{-1}$ takes values of invertible matrices \emph{a.e.}, once we find $P$ such that $P^\#\Omega = 0$, the Pfaff  equation~\eqref{pfaff, new} immediately follows. Hence, in loose terms (neglecting momentarily the local \emph{vs} global issues for the gauges and/or isometric immersions), we observe:
\begin{align*}
&\text{The existence of isometric immersion $\iota$ follows by showing that}\nonumber\\
&\text{the connection 1-form $\Omega$ is (locally) gauge-equivalent to the trivial connection.} \tag{$\clubsuit$}
\end{align*}

This is the starting point of our analytic estimates, which  rely crucially on  Rivi\`{e}re--Struwe \cite{rs} concerning the higher dimensional and codimensional harmonic maps with critical regularity in Morrey spaces. Gauge theoretic terminologies are explained in detail in \S\ref{sec: gauge} below.



\subsection{Organisation and remarks}
 
The remaining sections of the paper are organised as follows.

  \S\S\ref{sec: nomenclature}--\ref{sec: proof} are devoted to the proof of our Main Theorem~\ref{thm: main, new}. Notations and nomenclatures concerning geometry and function spaces shall be explained in \S\ref{sec: nomenclature}. Then, in \S\ref{sec: gauge} we introduce the Coulomb--Uhlenbeck gauge and its corresponding estimates in Morrey spaces, which constitute the key tool for our arguments. We conclude the proof for  Theorem~\ref{thm: main, new} in \S\ref{sec: proof}. Finally, as a corollary of the Main Theorem~\ref{thm: main, new}, we present a weak compactness theorem for slightly supercritical immersions in weak  Morrey spaces in \S\ref{sec: weak compactness}.


This paper is in continuation of \cite{li-new}, in which the first named author gives a detailed account for a proof of the fundamental theorem of submanifold theory in the subcritical regime, \emph{i.e.}, for $\iota \in W^{2,p}(\M, \R^{n+k})$ with $p>n=\dim\M$, hence presenting an alternative proof of S. Mardare and Szopos' results in \cite{m2, m3, sz} in the framework of Cartan formalism. It is also demonstrated in \cite{li-new} that a (non-gauge-equivalent) variant of Uhlenbeck's weak compactness theorem \cite{u} can be deduced from the nonlinear smoothing techniques developed in \cite{m2}.

The key analytic tool for our developments in this paper, \emph{i.e.}, the estimates in critical Morrey spaces for harmonic maps in arbitrary dimensions and codimensions \emph{\`{a} la} Rivi\`{e}re--Struwe \cite{rs}, has also been exploited in the energy-critical regimes of many other PDE systems in mathematical physics. See, \emph{e.g.},  Chen--Jost--Wang--Zhu \cite{cjwz} for Dirac-harmonic maps, and Guo--Liu--Xiang \cite{guo2} for a nonlinear Cosserat micropolar model in elasticity.

The regularity class of isometric immersions considered in this work --- $L^{p,n-p}_{2,w}$ for $2<p \leq n$ --- is somewhat stronger than the ``wild'', \emph{non-rigid} $C^{1,\gamma}$-isometric embeddings or immersions constructed by convex integration schemes, pioneered by Nash in the groundbreaking paper  \cite{n1} and recently further developed by De Lellis and Sz\'{e}kelyhidi Jr., among many others; \emph{cf}. \cite{ds} for a comprehensive survey. It should be pointed out that in all these later works involving convex integration,  curvatures (both intrinsic and extrinsic ones) cannot be defined --- in general, not even in the distributional sense. 

\section{Preliminaries}\label{sec: nomenclature}

In this section, we give a detailed account for the notations and nomenclatures used in this paper. Throughout, the symbol $C=C(a_1, a_2, \ldots)$ denotes that the constant $C$ depends on the parameters $a_1$, $a_2$, $\ldots$.

\subsection{Geometry}

Our notations on differential geometry are standard. We refer the reader to do Carmo \cite{d} for elements of differential geometry, to Chern \textit{et al} \cite{chern} and the recent exposition by Clelland \cite{clelland} for Cartan's moving frames.
 
The space of $m \times m$ matrices with real entries is designated as $\gl(m;\R)$, while $GL(m;\R)$ denotes the group of invertible matrices in $\gl(m;\R)$. Its subgroup $SO(m)$ consists of the orthogonal matrices in $GL(m;\R)$, and whose Lie algebra is $\so(m)$,  the space of skew-symmetric matrices. Also, ${\bf Id}$ denotes the identity matrix.

The vector bundle of $r$-fold exterior product of $T^*\M^n$ is written as $\bigwedge^r T^*\M$, whose sections $\sigma$ are the differential $r$-forms on $\R^n$. We write $\sigma \in \Omega^r(\M):=\G\left(\bigwedge^r T^*\M\right)$. Given a Lie algebra $\g$, we denote by 
$\Omega^r(\M; \g):=\G\left(\bigwedge^r T^*\M \otimes \g\right)$ the space of $\g$-valued $r$-forms over $\M$.

 For example, for $r=1$ and $\g \subset \gl(m; \R)$ a matrix subalgebra, any $\Theta \in \Omega^1(\M;\g)$ can be represented in local co-ordinates as below:
\begin{equation*}
\Theta= \Big({}^{[1]}\Theta^a_b, {}^{[2]}\Theta^a_b,\cdots,{}^{[n]}\Theta^a_b\Big)^\top, 
\end{equation*}
where ${}^{[i]}\Theta \equiv \left\{{}^{[i]}\Theta^a_{b}\right\}_{1 \leq a,b \leq m}: \M \to \g$ for each fixed $i \in \{1,\ldots,n\}$. In the sequel, we always take $m=n+k$.

The Sobolev spaces $W^{\ell,p}$ for fields of vectorfields, differential forms,  connections, matrix-valued differential forms, etc., are defined as usual. We write $\|\bullet\|_{W^{\ell,p}(U)}$ for the $W^{\ell,p}$-norm taken over $U \subset \M$.  The symbols $d$ and $d^*$ are the exterior differential and co-differential, respectively. The Laplace--Beltrami operator is $\Delta = dd^*+d^*d$, and the star $\star$ is reserved for the Hodge star operator, in contrast to the asterisk $*$ that denotes operator adjoints or matrix conjugates.

For a differential $r$-form $\alpha$ on $U \subset \M$, we write $\ttt\alpha$ and $\nnn\alpha$  for their tangential and normal traces on $\p U$, respectively. That is, $\ttt\alpha:=\jmath^\#\alpha$ and $\nnn\alpha:=\alpha-\ttt\alpha$, where $\jmath$ is the inclusion $\p U \emb \overline{U}$ and $\jmath^\#$ is the pullback operator under $\jmath$. These trace operators extend naturally to any $L^p$-differential form $\alpha$ if, in addition, $d\alpha$ or $d^*\alpha$ is also of $L^p$-regularity. See Schwarz \cite[p.27, Chapter 1, Eqs.~(2.25) and (2.26)]{sch}.


\subsection{Function spaces}
Let $U$ be a bounded smooth domain in $\R^n$. For $1 \leq p<\infty$, $L^p(U)$ denotes the standard Lebesgue space and $L_w^p(U)$ the weak $L^p$-space.  It holds that
\begin{align}\label{weak Lp embedding}
L^p(U) \subset L_w^p(U) \subset L^q(U) \qquad \text{ for } 1 \leq q<p.
\end{align}

The following $L^p$-based spaces, first appeared in Morrey's seminal work \cite{morrey} on regularity of harmonic mappings from 2D domains into manifolds, are now known as the Morrey spaces:

\begin{definition}
Let $1 \leq p<\infty$ and $0 \leq \lambda \leq n$. 
The Morrey space $L^{p, \lambda}(U)$ consists of functions $f \in L^p(U)$ that satisfy 
$$
\|f\|_{L^{p, \lambda}(U)} := \sup _{x \in U,\, 0<r\le \operatorname{diam}(U)} r^{-\lambda / p}\|f\|_{L^p\left(B_r(x) \cap U\right)}<\infty .
$$
The weak Morrey space $L_w^{p, \lambda}(U)$ consists of functions $f \in L_w^p(U)$ that satisfy 
$$
\|f\|_{L_w^{p, \lambda}(U)} := \sup _{x \in U,\, 0<r\le \operatorname{diam}(U)} r^{-\lambda / p}\|f\|_{L_w^p\left(B_r(x) \cap U\right)}<\infty .
$$
In addition, define inductively $f \in L_\ell^{p,\lambda}(U)$ by $\na f \in L_{\ell-1}^{p,\lambda}(U)$ for each $\ell = 1,2,3,\ldots$.
\end{definition}

When $\lambda=0$, $L^{p, 0}(U)=L^p(U)$ and $L_w^{p, 0}(U)=L_w^p(U)$. When $\lambda=n$, $L^{p, n}(U)=L^{\infty}(U)$ due to the Lebesgue differentiation theorem. In addition, 

\begin{lemma} \label{weak Morrey embedding}
The following embedding results hold:
\begin{enumerate}
\item
For any $1 \leq q \leq p<\infty$ and $0\le \theta , \sigma \le n$ such that $\frac{\theta-n}{q} \le \frac{\sigma-n}{p}$,\,one has $L^{p,\sigma}(U)\subset L^{q,\theta}(U)$.
\item
For any $1 \leq q<p<n$, one has $L_w^{p,n-p}(U) \subset L^{q,n-q}(U)$.
\item
For any $1 \leq q<p<\infty$ and $0 \le \lambda \le n$, one has $L_w^{p,\lambda}(U) \subset L^{q,\lambda}(U) \subset L_w^{q,\lambda}(U)$.
\end{enumerate}
\end{lemma}
\begin{proof}
It follows by applying Hölder's inequality to the Lorentz norms. See Grafakos \cite[Section 1.4]{Grafakos} and the recent work \cite{Gunawan}.
\end{proof}

We also recall the Hardy and BMO spaces:

\begin{definition}
\begin{enumerate}
\item
A function $f: U \subset \R^n \to \R$ is of bounded mean oscillation, denoted as $f \in  {\rm BMO}(U)$, if 
$$
[f]_{{\rm BMO}(U)} :=\sup _{x \in U,\, 0<r<\operatorname{diam}(U)} r^{-n / p} \left\|f-f_{r, x}\right\|_{L^p\left(B_r(x) \cap U\right)}<\infty.
$$
Here
$$
f_{r, x}=\frac{1}{\left|B_r\left(x\right) \cap U\right|} \int_{B_r\left(x\right) \cap U} f(y) \,\dd y .
$$
\item
The Hardy space $\mathcal{H}^p(\mathbb{R}^n)$ for $p \in ]0,\infty]$ consists of functions $f \in L^p(\mathbb{R}^n)$ such that 
$$\sup _{t>0}\left\{t^{-n}\left|\int_{\R^n} f(y) h\left(\frac{x-y}{t}\right)\,\dd y \right|\right\} \in L^p(\mathbb{R}^n),
$$
where $h$ is a fixed Schwartz function that satisfies $\int_{\R^n} h = 1$. This supremum is the Hardy space seminorm $\|f\|_{\mathcal{H}^p}$. 
\end{enumerate}

\end{definition}

Notice that  $\|f\|_{\mathcal{H}^p}$ depends on the choice of $h$, but different such choices all result in equivalent seminorms. The following properties for Hardy spaces are well-known (\cite{stein}):

\begin{lemma} \label{Hardy embedding}
\begin{enumerate}
\item
For $0< q \le p \le \infty$ we have $\mathcal{H}^p \subset \mathcal{H}^q$. The $\mathcal{H}^p$-seminorm increases in $p$.
\item
For $1 < p < \infty$, the Hardy space $\mathcal{H}^p$ is just $L^p$ with an equivalent norm.
\item
One has the duality $\left[\mathcal{H}^1(\R^n)\right]^* \cong {\rm BMO}(\R^n)$.
\end{enumerate}
\end{lemma}

A closely related function space to the Morrey space is the following:
\begin{definition}
The Campanato space ${\mathcal{L}^{p, \lambda}(U)}$ consists of functions $f$ such that 
$$
[f]_{\mathcal{L}^{p, \lambda}(U)} :=\sup _{x \in U,\, 0<r<\operatorname{diam}(U)} r^{-\lambda/ p} \|f-f_{r, x}\|_{L^p\left(B_r(x) \cap U\right)}<\infty.
$$
\end{definition}
For $\lambda=n$, the Campanato space $\mathcal{L}^{p,\lambda}(U)$ coincides with the BMO space on $U$, and for $n<\lambda\le n+p$ it coincides with the Hölder space $C^{0,\gamma}(U)$ with $\gamma =\frac{\lambda-n}{p}$. See \cite[Theorem~I.2]{camp}.

In \S\ref{sec: weak compactness} we shall make use of local Sobolev/Morrey spaces $L^{p}_\loc$,  $L^{q,\lambda}_{\loc}$, $W^{\ell,p}_\loc$, and $L^{q,\lambda}_{1,w,\loc}$, etc. For a function space $\X$ defined over $\M$, as usual we write $\X_\loc$ for the space of functions $f$ defined on $\M$ such that $f \1_V \in \X$ for every open subset $V \Subset \M$.

\section{Rudiments of Gauge Theory}\label{sec: gauge}

As explained in the introduction, the key to our developments is to study the isometric immersions problem via gauge theory, which allows us to solve the Pfaff  equation~\eqref{pfaff, new} and the Poincar\'{e} equation~\eqref{poincare, new}  for the connection 1-form $\Omega$ of slightly supercritical regularity. An essential ingredient for our proof of Theorem~\ref{thm: main, new} is  putting $\Omega$ into the \emph{Coloumb--Uhlenbeck gauge}, first introduced by Uhlenbeck in the seminal work \cite{u} and further explored by H\'{e}lein \cite{h} and Rivi\`{e}re \cite{r}, etc.

We refer to Wehrheim \cite{w}  for rudiments on gauge theory and a detailed exposition on Uhlenbeck's work \cite{u}. The crucial analytic results utilised in our paper are adapted from Rivi\`{e}re--Struwe \cite{rs}. Our developments also come in line with Litzinger \cite{l}, which exploits the gauge-theoretic approach to solving the Pfaff equation~\eqref{pfaff, new}, and hence showing the existence the isometric immersion, in the critical case for dimension 2.

\subsection{Gauge theory and isometric immersions}

Recall that in \S\ref{subsec: cartan} we have referred to $\Omega = \left\{\Omega^a_b\right\}_{1 \leq a,b \leq n+k}$, defined entry-wise via Equations~\eqref{Omega 1}--\eqref{Omega 3}, as the \emph{connection 1-form} of the Cartan formalism for immersed submanifolds. Indeed, it defines an affine connection on the frame bundle $$\pi: \mathcal{F}\left(T\M \oplus E\right) \longrightarrow \M,$$ which is a principal $G$-bundle over $\M$ with $G= SO(n+k)$: 
$$\Omega \in \hat{\Omega} +  \Omega^1(\M;\so(n+k))= \hat{\Omega} + \G\left(\M, T^*\M \otimes \so(n+k) \right).$$
Here $\hat{\Omega}$ is a background connection, and $\Omega$ lies in the affine space over $\g = \so(n+k)$-valued 1-forms on $\M$. 

Here and hereafter, to both simplify the notations and emphasise the global viewpoint, for the Cartan formalism for submanifold theory we write, schematically and for obvious reasons,
\begin{equation}\label{schematic, Omega}
\Omega = \begin{bmatrix}
\na & \two\\
-\two^\top & \na^E
\end{bmatrix}
\end{equation}
where $\na$, $\two$, and $\na^E$ are $n \times n$, $n \times k$, and $k \times k$ matrices with entries defined by Equations~\eqref{Omega 1}, \eqref{Omega 2}, and \eqref{Omega 3}, respectively. The second structural equation~\eqref{second structure eq} may be written compactly as
\begin{equation}\label{dOmega equation, new}
d\Omega + \Omega \wedge \Omega = 0,
\end{equation}
which is an identity in $\Omega^2(\M;\so(n+k))=\G\left(\M, \bigwedge^2 T^*\M \otimes \so(n+k) \right)$. The symbol $\Omega \wedge \Omega$ (also commonly denoted as $[\Omega \wedge \Omega]$ in the literature) is understood as  taking both wedge product on the exterior bundle factor $\bigwedge^\bullet T^*\M$ and Lie bracket on the Lie algebra factor $\so(n+k)$.

In local co-ordinates, Equation~\eqref{dOmega equation, new} reads:
\begin{equation*}
0 = \p_i {}^{[j]} \Omega^a_b - \p_j {}^{[i]} \Omega^a_b + \sum_{c=1}^{n+k} \left\{ {}^{[i]}\Omega^a_c {}^{[j]}\Omega^c_b -  {}^{[j]}\Omega^a_c {}^{[i]}\Omega^c_b \right\}
\end{equation*}
for all $1 \leq i,j \leq n$ and $1 \leq a,b \leq m=n+k$. Recall from \S\ref{subsec: cartan} that this is equivalent to the Gauss--Codazzi--Ricci equations for isometric immersions.  

As a connection on the frame bundle $\pi: \mathcal{F}\left(T\M \oplus E\right) \to \M$, the connection 1-form $\Omega$ has its \emph{curvature 2-form} defined by 
\begin{equation}\label{def, curvature}
\mathfrak{F}_\Omega := d\Omega+\Omega \wedge \Omega \in \Omega^2(\M; \so(n+k)).
\end{equation}
Thus, the second structural equation~\eqref{dOmega equation, new} expresses the \emph{flatness} of $\Omega$, which is  necessary for the existence of isometric immersion into a Euclidean space.

Furthermore, a gauge transform  on the frame bundle  is a mapping $$P: \M \longrightarrow G=SO(n+k),$$ which acts on $\Omega$ and $\mathfrak{F}_\Omega$ according to  Equations~\eqref{gauged connection 1-form} and \eqref{gauged curvature}, reproduced below: $$P^\#\Omega := P^{-1}dP + P^{-1}\Omega P$$ and $$P^\#\mathfrak{F}_\Omega = d\left(P^\#\Omega\right) + \left(P^\#\Omega\right)\wedge\left(P^\#\Omega\right).$$ 
The juxtapositions $P^{-1}dP$, $P^{-1}\Omega P$, and the like always designate matrix multiplication. As $\mathfrak{F}_\Omega \equiv 0$ by \eqref{dOmega equation, new}, the gauge-transformed connection is also flat, namely that 
\begin{equation}\label{flat, new}
P^\#\mathfrak{F}_\Omega \equiv 0.
\end{equation}

Indeed, one may deduce \eqref{flat, new} from the direct computations below, without resorting to general gauge-theoretic developments. Setting $$\Xi := P^{-1} dP + P^{-1}\Omega P,$$   we obtain that
\begin{align*}
Pd\Xi &= Pd\big(P^{-1} dP + P^{-1}\Omega P\big)\\
&= P\Big\{ d(P^{-1})\wedge dP + P^{-1}d(dP) +  d(P^{-1})\wedge\Omega P + P^{-1} d\Omega P - P^{-1}\Omega \wedge dP\Big\}.
\end{align*}
But $ddP=0$ and $Pd(P^{-1})=-(dP) P^{-1}$, so
\begin{align*}
Pd\Xi = -(dP) P^{-1} \wedge (dP + \Omega P) + (d\Omega) P -\Omega \wedge dP. 
\end{align*}
On the other hand, the definition of $\Xi$ yields that $dP = P\Xi - \Omega P$. This together with the second structural equation~\eqref{second structure eq}, namely that $d\Omega +\Omega \wedge \Omega = 0$, allows us to continue the previous equalities for $Pd\Xi$ as follows:
\begin{align*}
Pd\Xi &= -(dP) P^{-1} \wedge P\Xi -  \Omega \wedge \Omega P - \Omega \wedge (P\Xi-\Omega P)\\
&=-dP \wedge \Xi -  \Omega \wedge \Omega P - \Omega \wedge P \Xi + \Omega \wedge \Omega P\\
&= -dP\wedge \Xi - \Omega \wedge P\Xi \\
&= -(P\Xi -\Omega P) \wedge \Xi - \Omega \wedge P\Xi\\
&= - P\Xi \wedge \Xi.
\end{align*}
Thus $d\Xi +\Xi\wedge \Xi = 0$, which is tantamount to \eqref{flat, new} and expresses the flatness of the gauge-transformed connection 1-form $P^\#\Omega$.




\subsection{Key Lemma on the Coulomb--Uhlenbeck gauge}

The key analytic tool for our proof of Theorem~\ref{thm: main, new} is the lemma below. Our presentation follows Rivi\`{e}re--Struwe \cite[Lemma~3.1]{rs}.
\begin{lemma}\label{lem: riviere}
Let $U\subset\R^n$ be a smooth bounded domain, $k$ be a natural number, and $q\in ]2,n]$. There exists some $\e_{\rm Uh} = \e(U,n,k,q)>0$ such that the following holds. Assume that $\Omega \in L^{q,n-q}\left(U;\so(n+k)\otimes\bigwedge^1\R^n\right)$ satisfies $\|\Omega\|_{L^{q, n-q}} \le \e$ for some $\e \in \left]0,\e_{\rm Uh}\right[$. Then one can find $P \in {L_1^{q, n-q}}(U;SO(n+k))$ and   $\xi \in {L_1^{q, n-q}}\left(U;\so(n+k)\otimes\bigwedge^{n-2}\R^n\right) $ such that
\begin{subequations} 
\begin{align}
P^{-1} d P+P^{-1} \Omega P & =\star d \xi & & \text { in } U, \label{xxa, new}\\
d(\star \xi) & =0 & & \text { in } U, \label{xxb, new}\\
\xi & =0  & & \text { on } \partial  U. \label{xxc, new}
\end{align}
\end{subequations} 
Moreover, we have the estimate
\begin{equation} \label{zz, new}
\|d P\|_{L^{q, n-q}}+\|d \xi\|_{L^{q, n-q}} \leq C\|\Omega\|_{L^{q, n-q}} \le C\varepsilon
\end{equation}
for some $C=C(n,k,q)$. 
\end{lemma}

The proof is a straightforward adaptation of the arguments for  
\cite[Lemmas~3.1 and 4.1]{rs}: one merely replaces the $L^{2,m-2}$-estimates for $\Omega$ therein by ${L^{q, n-q}}$-estimates. For the sake of completeness, we outline a proof of Lemma~\ref{lem: riviere}, highlighting only the differences from \cite{rs}.


\begin{proof}[Sketched Proof of Lemma~\ref{lem: riviere}]
As in Uhlenbeck \cite{u} and Rivi\`{e}re--Struwe \cite{rs}, we solve for $P$ and $\xi$ with a little more stringent regularity conditions. More precisely, assume furthermore that $\Omega$ has finite $L^{q, n-q+\alpha}$-norm for some $\alpha>0$ with $q+\alpha \leq n$. This is eligible since the mollification of $\Omega$ (not relabelled) yields a smooth connection 1-form over the bounded domain $U$ with decreased $L^{q,n-q}$-norm, hence preserving the hypothesis  of the theorem.  Let us seek $P \in {L_1^{q, n-q+\alpha}}(U;SO(n+k))$ and $\xi \in {L_1^{q, n-q+\alpha}}\left(U;\so(n+k)\otimes\bigwedge^{n-2}\R^n\right)$ satisfying \eqref{xxa, new}--\eqref{xxc, new}, as well as the estimate \begin{equation} \label{yy, new}
\|d P\|_{L^{q, n-q+\alpha}}+\|d \xi\|_{L^{q, n-q+\alpha}} \leq C\|\Omega\|_{L^{q, n-q+\alpha}},
\end{equation}
in addition to \eqref{zz, new}. The point is that $P \mapsto P^{-1}$, the  matrix inversion, is a smooth mapping on ${L_1^{q, n-q+\alpha}}(U;SO(n+k))$ but \emph{not} on ${L_1^{q, n-q}}(U;SO(n+k))$. We need to go ``slightly subcritical'' to ensure that the gauge transforms act continuously on connection 1-forms.

We proceed by the method of continuity. Put
$$
\mathcal{V}_{\varepsilon}^\alpha:=\left\{\Omega \in L^{q,n-q+\alpha}\left(U;\so(n+k)\otimes\bigwedge^1\R^n\right) \text { with } \|\Omega\|_{L^{q, n-q}} \le \varepsilon \right\},
$$
which is a star-shaped hence path-connected set. We show that the subset
$$
\mathcal{U}_{\varepsilon}^\alpha:=\Big\{\Omega \in \mathcal{V}_{\varepsilon}^\alpha:\, \text {there exist $P$ and $\xi$ satisfying \eqref{xxa, new}--\eqref{xxc, new},  \eqref{zz, new}, and \eqref{yy, new}}\Big\}
$$
is both open and closed. It is clear that $\mathcal{U}_{\varepsilon}^\alpha$ contains the zero connection.

{\bf Closedness:} Given $\Omega \in \mathcal{U}_{\varepsilon}^\alpha$, we first obtain by mollification $\Omega_\delta \in C^\infty$ with $\|\Omega_\delta\|_{L^{q,n-q}}\leq C_0\e$ for all $\delta>0$ sufficiently small; here $C_0=C_0(n,k,q)$ is a uniform constant. Applying the arguments as for classical Uhlenbeck gauges \cite{u,w}, we obtain $P_\delta$ and $\xi_\delta$ associated to $\Omega_\delta$ that satisfy \eqref{xxa, new}--\eqref{xxc, new},  \eqref{zz, new}, and \eqref{yy, new}. That $P_\delta \in {L_1^{q, n-q+\alpha}}(U;SO(n+k))$ and that $P_\delta$, $P_\delta^{-1}$ take values in the compact group $SO(n+k)$ ensure the validity of  passage of limit $\delta \to 0$.

{\bf Openness:}
As $P \in {L_1^{q, n-q+\alpha}}$, we have $\nabla P \in {L^{q, n-q+\alpha}}$ and hence by the Poincaré inequality $P \in {\mathcal{L}^{q, n+\alpha}}$. Embedding theorems for Campanato spaces lead to $P \in C^{0,\alpha/q}$ (see the end of \S\ref{sec: nomenclature}), which justifies that the matrix inversion $P \mapsto P^{-1}$ from ${L_1^{q, n-q+\alpha}}(U;SO(n+k))$ to itself is a smooth mapping between Banach manifolds. This allows us to conclude, following the implicit function theorem arguments for \cite[Lemma 2.8]{u}, that  for every $\tilde{\Omega} \in L^{q,n-q+\alpha}\left(U;\so(n+k)\otimes\bigwedge^1\R^n\right)$  such that $\left\|\tilde{\Omega}-\Omega\right\|_{L^{q, n-q+\alpha}}$ is sufficiently small, there exist some  $P_{\tilde{\Omega}} \in {L_1^{q, n-q+\alpha}}(U;SO(n+k))$ and $\xi_{\tilde{\Omega}} \in {L_1^{q, n-q+\alpha}}\left(U;\so(n+k)\otimes\bigwedge^{n-2}\R^n\right)$ satisfying the PDE \eqref{xxa, new}--\eqref{xxc, new}. The verification that $\left(P_{\tilde{\Omega}}, \xi_{\tilde{\Omega}}\right)$ satisfy the bounds \eqref{zz, new} and \eqref{yy, new} follows essentially from the proof of  \cite[Lemma 4.2]{rs}, for which we replace all occurrences of $L^{2,m-2}$ therein by $L^{q,n-q}$; $2<q<n$. This is valid by virtue of the estimate:
$$\left[\xi_{\tilde{\Omega}}\right]_{{\rm BMO}(U)} \leq C(n,q,k) \left\|d \xi_{\tilde{\Omega}}\right\|_{L^{q, n-q}}$$  (see \cite[p.456]{rs}; noticing here that $d(\star \xi_{\tilde{\Omega}})=0$), as well as the standard elliptic estimates, \emph{e.g.}, in Giaquinta \cite[Theorem III.2.2]{Giaquinta}. Hence the openness of $\mathcal{U}_{\varepsilon}^\alpha$ follows.
\end{proof}

\section{Proof of the Main Theorem}\label{sec: proof}
We are now at the stage of proving  Theorem~\ref{thm: main, new}, reproduced below:

\begin{theorem*}
Let $(\M^n,g)$ be a simply-connected closed Riemannian manifold; $n \geq 3$. Suppose that $\mathfrak{S} \in L^{p,n-p}_w$ with $2<p \leq n$ is a weak solution to the Gauss--Codazzi--Ricci equations on $\M$ with arbitrary codimension $k \geq 0$. There exists an isometric immersion $\iota: (\M,g) \emb (\R^{n+k},\delta)$ in the regularity class $\X = L^{p,n-p}_{2,w}$ whose extrinsic geometry coincides with $\mathfrak{S}$. Moreover, $\iota$ is unique modulo Euclidean rigid motions in $\R^{n+k}$ outside null sets. 
\end{theorem*}

\begin{proof}[Proof of Theorem~\ref{thm: main, new}]

The proof is divided into four steps below.

\smallskip
\noindent
{\bf Step 1.} We first reduce the assertion to a local problem. 

As in Tenenblat \cite{ten} (\emph{cf.} the exposition in \cite[\S 3]{li-new} for details), the isometric immersion $\iota: (\M,g) \emb (\R^{n+k},\delta)$ can be constructed by solving consecutively a Pfaff  equation~\eqref{pfaff, new} and a Poincar\'{e} equation~\eqref{poincare, new}, assuming their  compatibility equations --- the second and the first structural equations~\eqref{second structure eq} and \eqref{first structure eq}, respectively.

For the Pfaff equation, for a sufficiently small number $v_0>0$ to be specified later, we consider an arbitrary ball $\ball \subset \M$ with ${\rm Volume}(\ball) \leq v_0$. By assumption of the theorem and the embedding $L^{p,n-p}_w \emb L^{q,n-q}$ for $2 <q<p \leq n$ (see Lemma~\ref{weak Morrey embedding}),  we know that the weak solution $\mathfrak{S}$ to the Gauss--Codazzi--Ricci equations has $L^{q, n-q}$-regularity. In view of Lemma~\ref{lem: riviere} (with $U=\ball$ therein), we then obtain a  gauge $P \in {L_1^{q, n-q}}(\ball;SO(n+k))$ and a  potential $\xi \in {L_1^{q, n-q}}\left(\ball;\so(n+k)\otimes\bigwedge^{n-2}\R^n\right)$.

Our goal is to prove the identity
\begin{equation}\label{to show, new}
\Xi := P^\#\Omega\big|_\ball = P^{-1} dP + P^{-1}\Omega P\big|_\ball = 0
\end{equation}
and the uniqueness of such $P$. Eq.~\eqref{to show, new} expresses that the local gauge $P$ on $\ball$ transforms the connection 1-form $\Omega$ to the trivial one. From \eqref{to show, new} one may immediately deduce the Pfaff equation~\eqref{pfaff, new}, namely that $d\Omega+P\Omega =0$ on $\ball$.

 Due to the local nature of the arguments above, we may take $\ball$ = Euclidean ball. Then, for the simply-connected closed manifold $\M$, a monodromy argument (see \cite[\S 3, Proof of Lemma~0.9]{li-new} for details) allows us to patch together the local gauges over an open covering for $\M$ by small balls, thus obtaining   a \emph{global} gauge $P \in L^{q,n-q}_1(\M; SO(n+k))$ which satisfies the Pfaff equation~\eqref{pfaff, new}. A limiting argument yields that $P \in L^{p,n-p}_{1,w}(\M; SO(n+k))$.  

With $P$ at hand, one may solve the Poincar\'{e} equation~\eqref{poincare, new}, namely $d\iota = \omega P$, to obtain the global isometric immersion $\iota \in L^{p, n-p}_{2,w}(\M; \R^{n+k})$. The unique solubility for \eqref{poincare, new} in $L^s$ ($s\geq 1$) on simply-connected domains, and hence in weak Morrey spaces as in our case, follows from a direct limiting argument subject to the first structural equation~\eqref{first structure eq}. See \cite{m+} for   Poincar\'{e} equation with weak regularity and \cite{ten} for checking the compatibility equation~\eqref{first structure eq} and solving for $\iota$.

\smallskip
\noindent
{\bf Step 2.} Now we embark on the proof for \eqref{to show, new}. Notice by  Lemma~\ref{lem: riviere}, \eqref{xxa, new} that $$\Xi = \star d \xi$$ for some $\xi \in {L_1^{q, n-q}}\left(\ball;\so(n+k)\otimes\bigwedge^{n-2}\R^n\right)$. In view of \eqref{gauged curvature}, $\xi$ satisfies the PDE:
\begin{equation}\label{xi eq, new}
\begin{cases}
d\star d\xi +\star d\xi \wedge \star d\xi = 0 \qquad \text{in } \ball,\\
\xi = 0 \qquad \qquad \text{on } \p\ball.
\end{cases}
\end{equation}
The boundary condition is taken from \eqref{xxc, new} and, in view of \eqref{xxb, new}, this PDE is elliptic.

We proceed with an energy estimate. Since $\xi \in L_1^{q, n-q} \emb \dot{W}^{1,2}$ for $q>2$ on $\ball \subset \R^n$ (see Lemma~\ref{weak Morrey embedding}), we may integrate $\xi$ against \eqref{xi eq, new} and deduce via the Stokes' theorem that
\begin{align}\label{L2, first estimate, new}
 \left\|\star d\xi\right\|^2_{L^2\left(\ball\right)}= \left\| d\xi\right\|^2_{L^2\left(\ball\right)} = -\int_\ball \xi \wedge \star d\xi \wedge \star d\xi.
\end{align}
Indeed,  it holds by \cite[p.60, Proposition~2.1.2]{sch} that
\begin{align*}
\int_\ball \xi \wedge \left(d\star d\xi\right) =  \left\|\star d \xi\right\|^2_{L^2\left(\ball\right)} +  \int_{\p \ball} \ttt \left(\star d \xi\right) \wedge \star \nnn \left(\star \xi\right),
\end{align*}
where the boundary term vanishes since $\star \nnn(\star\xi)=\star\star\ttt\xi=\pm \ttt \xi$ (\cite[p.27, Proposition~1.2.6]{sch}), while $\xi \big|_{\p\ball}=0$ by the boundary condition in \eqref{xi eq, new}.

Then, by $\mathcal{H}^1$-BMO duality, we infer from \eqref{L2, first estimate, new} that 
\begin{equation}\label{first bound, new}
\left\| d\xi\right\|^2_{L^2\left(\ball\right)} \le C_1 [\xi]_{BMO(\ball)}  \left\| \star d\xi \wedge \star d\xi \right\|_{\mathcal{H}^1(\R^n)}.
\end{equation}
Here and hereafter we identify $\xi$, without relabelling, with its extension-by-zero on $\R^n$ that lies in ${L_1^{q, n-q}}\left(\R^n;\so(n+k)\otimes\bigwedge^{n-2}\R^n\right)$. This is eligible due to the boundary condition $\xi=0$ on $\p\ball$, hence circumvents the technical complications arising from the Hardy/BMO space theory on bounded domains. Here $C_1=C(n,k)$.

To control the BMO-seminorm of $\xi$, we bound
\begin{align*}
 [\xi]_{{\rm BMO}(\ball)} &\leq C_2 \|d \xi\|_{L^{2, n-2}(\ball)}\\
 &\leq C_3\|d \xi\|_{L^{q, n-q}(\ball)} \\
 &\leq  C_4\|\Omega\|_{L^{q, n-q}(\ball)},
\end{align*}
where the first line follows from \emph{e.g.}, \cite[Section 3]{rs}, the second line from Lemma~\ref{weak Morrey embedding}, and the final one from Lemma~\ref{lem: riviere}, \eqref{zz, new}. The constant $C_4=C(n,p,k)$. 

Fix an arbitrarily small $\e'>0$ such that $$C_4 \e' < \e_{\rm Uh}$$ in Lemma~\ref{lem: riviere}. As $\Omega \in L^{p,n-p}_{w}$, we may take $\ball$ sufficiently small (depending on $\e'$) such that 
\begin{equation}\label{small xi in BMO}
[\xi]_{{\rm BMO}(\ball)}  \leq C_4 \e'.
\end{equation}

It now remains to show that 
\begin{equation}\label{comp comp, new}
\left\| \star d\xi \wedge \star d\xi \right\|_{\mathcal{H}^1(\R^n)} \le C_5\left\| d\xi\right\|^2_{L^2\left(\ball\right)}
\end{equation}
for $C_5=C(n,k)$. Assuming this,  we can infer from \eqref{first bound, new}, \eqref{small xi in BMO} and \eqref{comp comp, new} that $$\left\| d\xi\right\|^2_{L^2\left(\ball\right)} \leq C_4C_5\e'\left\| d\xi\right\|^2_{L^2\left(\ball\right)}.$$ Thus, by choosing $\e' = (2C_4C_5)^{-1}$, we immediately obtain $\Xi = \star d\xi=0$. This yields \eqref{to show, new} and hence complete the proof of the theorem.

\smallskip
\noindent
{\bf Step 3.} To establish \eqref{comp comp, new}, we resort to the theory of \emph{compensated compactness} pioneered by Murat \cite{murat} and Tartar \cite{t1,t2}. We  utilise the harmonic analytic version of the  compensated compactness theory developed in the seminal work \cite{clms} \emph{\`{a} la} Coifman--Lions--Meyer--Semmes. 

Indeed, we may recast \cite[p.258, III.2, ``Variants and more examples'']{clms} into the form of the wedge product theorem  by Robbin--Rogers--Temple \cite{rrt} as follows ($\overline{p}'$ denotes the H\"{o}lder conjugate of $\overline{p}$, namely that $\overline{p}'=\frac{\overline{p}}{\overline{p}-1}$): 
\begin{quote}
Assume $E$ and $B$ are  $L^{\overline{p}}$- and $L^{\overline{p}'}$-differential forms over $\R^n$, respectively, such that $dE \in W^{-1,r}$ and $dB \in W^{-1,s}$ for some $r>{\overline{p}'}$ and $s > {\overline{p}}$. Then $E\wedge B \in \mathcal{H}^1(\R^n)$. In addition, $\|E \wedge B\|_{\mathcal{H}^1} \leq C_6 \|E\|_{L^{\overline{p}}} \|B\|_{L^{\overline{p}'}}$ for some $C_6=C_6(\overline{p},n,\deg(E),\deg(B))$.
\end{quote}
For our purpose, we take $\overline{p}=2$ and $r=s=q$ where $2<q<p\leq n$ (recall that $\Omega \in L^{p,n-p}_w \emb L^{q,n-q} \emb L^q$ by Lemma~\ref{weak Morrey embedding}), and take $E=B=\Xi$ where $\Xi = \star d \xi$ as in \eqref{xxa, new}, with $\xi$ again identified with its extension-by-zero on $\R^n$.

 Thus \eqref{comp comp, new} follows, which completes the proof for the existence of isometric immersion $\iota$.

\smallskip
\noindent
{\bf Step 4.} Finally, we prove that the solution $P$ to the Pfaff equation~\eqref{pfaff, new} is unique. In fact, the uniqueness holds in $W^{1,2}(\M; SO(n+k))$, which is a larger space than $L_{1,w}^{p,n-p}(\M; SO(n+k))$ for $2<p\leq n$ (Lemma~\ref{weak Morrey embedding}). Indeed, let $P, \hat{P} \in W^{1,2}(\M;SO(n+k))$ both be solutions to \eqref{pfaff, new}. Using the identity $d\hat{P}^{-1}=-\hat{P}^{-1}\left(d\hat{P}\right) \hat{P}^{-1}$ and the Pfaff equation~\eqref{pfaff, new} for $\hat{P}$, we deduce that
\begin{align*}
\hat{P}d\left(\hat{P}^{-1}P\right) &= -\left(d\hat{P}\right)\hat{P}^{-1}P + dP\nonumber\\
&= \left(\Omega \hat{P}\right)\hat{P}^{-1}P -\Omega P \nonumber\\
&= 0.
\end{align*}
Hence $\hat{P}^{-1}P$ equals, in the distributional sense, a constant matrix in $SO(n+k)$.

The proof is now complete.  \end{proof}

\begin{remark}
Schikorra \cite{s} showed that gauge transforms can be found via variational approaches combined with H\'{e}lein's  moving frames, hence simplifying various arguments in \cite{r, rs}. Desired estimates were obtained in \cite{s} in the framework of Sobolev spaces, but we expect that they can be adapted to the setting of Morrey spaces. It would be interesting to investigate if the approach in \cite{s} sheds new lights on the Pfaff equation~\eqref{pfaff, new} and the Gauss--Codazzi--Ricci system.
\end{remark}

\section{Weak compactness of  immersed submanifolds} \label{sec: weak compactness}

Finally, we present a corollary of our main Theorem~\ref{thm: main, new}, namely the slightly supercritical fundamental theorem of submanifold theory, as Theorem~\ref{thm: weak rigidity} below. As a weak compactness theorem for immersed submanifolds in Euclidean spaces, it ascertains that a family of uniformly bounded $L^{p,n-p}_{2,w,\loc}$-immersions $\left\{\iota^\e\right\}$ (with $2 <p \leq n$ as in Theorem~\ref{thm: main, new}) of an $n$-dimensional manifold $\M$ --- \emph{subject to a nondegeneracy condition for the induced metrics}; see Remark~\ref{rem} --- converges weakly to an immersion, whose induced intrinsic and extrinsic geometries are both limiting points of those of $\left\{\iota^\e\right\}$ in natural topologies. Due to the local nature of this result, we may state it without loss of generality on a Euclidean domain $\M \subset \R^n$.

Our formulation for Theorem~\ref{thm: weak rigidity}  follows Litzinger \cite[Theorem~4]{l} and generalises it  to arbitrary dimensions and codimensions. See also \cite{cl, cm, li-new} for the analogous result in the subcritical case, \emph{i.e.}, for $W^{2,p}_\loc$ isometric immersions with $p>n=\dim\M$. Let us also mention Chen--Slemrod--Wang \cite{csw}, which established the weak continuity of Gauss--Codazzi--Ricci equations for weak solutions in $L^s_\loc$ for any $s>2$, regardless of dimension and codimension. This nevertheless is purely a PDE result: by now it remains unknown if local isometric immersions can always be found, given $L^s_\loc$-solutions to the Gauss--Codazzi--Ricci equations for general $s>2$.

\begin{theorem}\label{thm: weak rigidity}
Let $\left\{\iota^\e\right\}_{\e>0}$ be  uniformly bounded $L^{p,n-p}_{2,w,\loc}$-immersions of an $n$-dimensional domain $\M \subset \R^n$ into $\R^{n+k}$, where $2<p\leq n$. Assume the following conditions:
\begin{enumerate}
\item
for any compact subset $\mathcal{K} \Subset \M$, there are constants $0<c \leq C<\infty$ (depending possibly on $\mathcal{K}$) such that for almost every $x \in \mathcal{K}$, one has
\begin{align}\label{imm, new}
c \leq \Big\{\text{eigenvalues of the matrix } g^\e(x):=d\iota^\e(x)\otimes d\iota^\e(x)\Big\} \leq C;
\end{align}
\item
for each $j \in \{1,2,\,\ldots,n\}$ and $\e>0$, the first derivative $\p_j \iota^\e$ lies in $L^{p,n-p}_{1,w,\loc} \cap L^\infty_\loc$. 
\end{enumerate}
 
Then, modulo subsequences,  $\{\iota^\e\}$ converges  weakly in $L^{p,n-p}_{2,w,\loc}$ to an immersion $\overline{\iota}$, whose induced metric $\overline{\iota}^\#\delta$ is a 
limiting point of $g^\e$ in $L^{p,n-p}_{1,w,\loc}$, and whose extrinsic geometry (\emph{i.e.}, second fundamental form and normal connection) is a 
limiting point of that of $\iota^\e$ in $L^{p,n-p}_{w,\loc}$. 
\end{theorem}



\begin{proof}[Proof of Theorem~\ref{thm: weak rigidity}]

For each $\e>0$, let $\Omega^\e$ be the connection $1$-form of the Cartan formalism for Euclidean submanifolds corresponding to $\iota^\e$ defined via Eqs.~\eqref{Omega 1}--\eqref{Omega 3}. Here,  by construction, 
\begin{align*}
\iota^\e : \left(\M, g^\e\right) \longrightarrow \left(\R^{n+k},\delta\right)
\end{align*}
are isometric immersions. As in \eqref{schematic, Omega} we write schematically
$$\Omega^\e = \begin{bmatrix}
\na^{g^\e}&\two^\e\\
-\left(\two^\e\right)^\top&\na^{\perp,\e}
\end{bmatrix},$$
where $\na^{g^\e}$, $\two^\e$, and $\na^{\perp,\e}$ are the coefficients of the Levi-Civita connection of $g^\e$, the second fundamental form of $\iota^\e$, and the normal connection of $\iota^\e$, respectively.

In view of the definition of $\Omega^\e$ and the $L^\infty$-boundedness for $g^\e$, it holds that  $\Omega^\e \weak \overline{\Omega}$ weakly in $L^{p,n-p}_{w,\loc}$, after passing to subsequences if necessary. From Lemma~\ref{lem: riviere},  \eqref{zz, new} and  Lemma~\ref{weak Morrey embedding}, one infers that $\left\{dP^\e\right\}$ is bounded in $L^2_\loc$, where $P^\e$ is the  Coulomb--Uhlenbeck gauge  associated to $\Omega^\e$. As $P^\e$ takes values in the compact group $SO(n+k)$, on a local chart $U \subset \M$, we have the weak convergence $P^\e \weak \overline{P}$ in $W^{1,2}_\loc$ for some $\overline{P} \in L_{1,w,\loc}^{p,n-p}(U; SO(n+k))$. By the Rellich--Kondrachov Lemma, we thus have 
\begin{equation}\label{strong conv, L^2}
\text{$P^\e \longrightarrow \overline{P}$ strongly in $L^2_\loc$}.
\end{equation}

In addition, by virtue of Theorem~\ref{thm: main, new} (especially Step~1 of the proof) and passing to a smaller chart $U \subset \M$ if necessary, we obtain the Pfaff system: $$d P^\e + \Omega^\e P^\e = 0.$$ 
Then, thanks to the strong convergence in \eqref{strong conv, L^2}, we can  pass to the limits to deduce that $$d\overline{P} + \overline{\Omega}\, \overline{P} = 0$$ in the sense of distributions.

Now, thanks to the proof of Theorem~\ref{thm: main, new} (see Step~1 therein, via the Poincar\'{e} system~\eqref{poincare, new}), the existence of an $L^{p,n-p}_{2,w,\loc}$-isometric immersion  $\overline{\iota}: \left(U, g^\e\right) \to \left(\R^{n+k},\delta\right)$ can be deduced from the solubility of $d\overline{P} + \overline{\Omega}\, \overline{P} = 0$. By the uniqueness of distributional limits, we have that $\iota^\e \weak \overline{\iota}$ weakly in $L^{p,n-p}_{2,w,\loc}$ modulo subsequences. The convergence of  metrics and extrinsic geometries follows from the previous arguments.   \end{proof}

\begin{remark}\label{rem}
The nondegeneracy assumption, \emph{i.e.}, the existence of $c>0$ in \eqref{imm, new}, is necessary for Theorem~\ref{thm: weak rigidity}. It ensures that the limits of $\iota^\e$ do not pinch to non-immersions. See Langer \cite[p.227]{langer} and Li \cite[\S 2]{li} for examples of such pinching behaviours of smooth immersed hypersurfaces $\M^n \emb \R^{n+1}$ ($n\geq 2$) with uniformly $L^n$-bounded second fundamental forms.
\end{remark}

\begin{remark}
Analogues of Theorems~\ref{thm: main, new} and \ref{thm: weak rigidity} in this paper  for isometric immersions of semi-Riemannian manifolds (with fixed signatures) are expected to be valid. See \cite{cl+} for the semi-Riemannian version of the fundamental theorem of submanifold theory and the weak compactness result for $W^{2,p}_\loc$-isometric immersions of $\M^n$ into semi-Euclidean spaces; $p>n$.
\end{remark}

\bigskip
\noindent
{\bf Acknowledgement}. SL thanks Armin Schikorra and Reza Pakzad for many insightful discussions on Gauss--Codazzi--Ricci equations, isometric immersions, and compensated compactness. Both authors are indebted to Yuning Liu for organising the working seminar on harmonic maps at NYU-Shanghai in Fall 2023, and for giving nice talks on  \cite{rs}.

The research of SL is supported by NSFC Projects 12201399 and 12331008, Young Elite Scientists Sponsorship Program by CAST 2023QNRC001, National Key Research $\&$ Development Program $\#$SQ2023YFA1000024, and  Shanghai Frontier Research Institute for Modern Analysis.

\bigskip
\noindent
{\bf Competing Interests Statement}. We declare that there is no conflict of interests involved.

\end{document}